\documentclass[a4paper,11pt]{article}
\usepackage{amsmath,amssymb,amsthm,a4wide,color}
\usepackage[utf8]{inputenc}
\usepackage[T1]{fontenc}
\usepackage{authblk,mathrsfs,empheq}
\usepackage{mathtools}
\usepackage{algorithm,algpseudocode}
\usepackage{graphicx,subcaption}
\usepackage{esint}
\usepackage{multirow}
\usepackage[colorlinks=true,citecolor=blue]{hyperref}
\usepackage{bm}
\usepackage{booktabs}
\usepackage{ragged2e}
\usepackage{caption}
\usepackage{subcaption}
\usepackage{pdfsync}
\usepackage[capitalise,noabbrev]{cleveref}
\usepackage{todonotes}
\input{macro}

\title{\bf Coarse space preconditioning for \\
  Generalized Optimized Schwarz Methods.\\
  Part I: continuous case.}
\author[1]{Xavier Claeys}

\affil[1]{{\small POEMS, CNRS, Inria, ENSTA, Institut Polytechnique de
    Paris, 91120 Palaiseau, France} }
\date{}

\begin{document}

\maketitle

\begin{abstract}  
  \noindent 
  The Generalized Optimized Schwarz Method (GOSM) originally proposed in
  [Claeys, 2021] is a variant of Depr\'es algorithm, a domain decomposition
  strategy for the solution of harmonic wave propagation problems. It
  imposes transmission conditions through interfaces by means of a
  non-local exchange operator. Conducting our analysis at the continuous
  level, in an infinite dimensional setting, we propose a coarse space
  construction for the preconditioning of the GOSM formulation, and provide
  estimates for the convergence of GMRes applied to the preconditioned
  equation.
\end{abstract}

\section*{Introduction}
Large scale simulation of wave propagation problems in harmonic regime
remains a major computational challenge. Although decisive progress has
been achieved for positive definite equations in recent years by means of
multi-level domain decomposition \cite{zbMATH06303461}, harmonic wave
propagation does not properly belong to this class of problems and still
lies at the center of an active research area.

One of the most established domain decomposition (DDM) strategies for wave
propagation rests on an algorithm proposed by Bruno Després
\cite{Despres1991DDM} that subdivides the computational domain into
non-overlapping parts, and solves the wave equation separately in each
subdomain, maintaining a coupling of adjacent subdomains by swapping
ingoing/outgoing Robin unknown traces to enforce transmission
conditions. We shall refer to this approach as \textit{Després
  algorithm}. Based on the observation that the choice of impedance factor
for Robin traces has an important impact on the convergence speed of
Després algorithms, many variants of this approach have been proposed since
then, most of which differ in the way of choosing the impedance factor,
trying to tune it so as to improve convergence of the overall method. We
shall refer to this larger class of methods as \textit{Optimized Schwarz
  Methods} (OSM).

The present contribution focuses on a variant of the Després algorithm,
referred to as Generalized Optimized Schwarz Method (GOSM) developped in a
series of contributions \cite{MR4665035,MR4648528,MR4433119} as a variant
of Després algorithm robust to meshwidth refinement and insensitive to the
presence of cross points in the subdomain partition. It shares many
similarities with pre-existing OSM strategies but departs from them in the
way transmission conditions are imposed. While OSM systematically encodes
transmission conditions by simple (local) swaps of Robin traces at each
interface of the subdomain partition, GOSM relies on a non-local exchange
operator.

\quad\\ Although, by construction, GOSM is robust to mesh refinement, the
later is not the only factor that may harm the convergence of linear
solvers. A deterioration of convergence may also stem from an increase of
the number of subdomains or from (quasi)resonance phenomena inherited from the
underlying continuous wave boundary value problem. A coarse space
correction appears as a natural way of dealing with these issues.

The recent literature on multi-level overlapping Schwarz methods has seen
the developement of scalable and accurate coarse spaces that can cope with
resonance phenomena
\cite{zbMATH08143001,arXiv:2409.06533,UnifiedMa2025,ma2025mmr2l,arXiv:2404.02758,arXiv:2605.31552}. Such
constructions remain valid for wave propagation problems, although the
offline cost of contruction of the coarse spaces themselves may represent a
significant computational bottleneck. These constructions hold for domain
decomposition with overlap, and it is not clear how to transfer them to
non-overlapping DDM.

There also exists a large literature on coarse space corrections for
substructuring non-overlapping domain decomposition, see
\cite{MR3013465,zbMATH06534518,MR2104179}
and references there in. Although most of such strategies are designed for
positive definite problems, a few contributions also deal with
Helmholtz problems 
\cite{zbMATH05729644,zbMATH01455844,zbMATH01488464}. Existing approaches
are not applicable to GOSM though, due to its specific structure.

\quad\\ The present contribution develops a coarse space preconditioning
strategy for GOSM.  Let us recall that OSM and GOSM are systematically
formulated as equations of the form $(\Id + \Pi \mS)\bu = \bff$ where $\bu
= (u_1,\dots,u_N),\bff = (f_1,\dots, f_N)$ are collections of trace
functions defined on the boundaries of subdomains, the operator $\Pi$ is an
involution $\Pi^2 = \Id$ that achieves a coupling between subdomains by
reformulating transmission conditions through interfaces, and $\mS =
\mrm{diag}(\mS_1,\dots, \mS_N)$ is a block diagonal operator that solves
the wave equation in each subdomain. An important observation of the
present contribution is that, in the case of GOSM, the local scattering
maps only differ from a multiple of the identity by a compact perturbation
i.e. $\mS - i\Id$ is compact. At the numerical level, this means that $\mS
- i\Id$ is low-rank compressible. This is a specific feature of GOSM that
is not verified by conventional OSM strategies.  We show how this
observation can be exploited to derive an inverse of the operator $\Id
+\Pi\mS$ up to a compact perturbation which, ultimately, will deliver a
parametrix. All the analysis presented here is at the continuous level,
adopting the same perspective as \cite{MR4665035,zbMATH07381643}. In a
forthcoming contribution, we develop a discrete counterpart detailing how
to exploit this analysis to obtain a concrete preconditioner.

\quad\\ In our analysis, we take GMRes \cite{Saad2003IMS,zbMATH03967793} as
our reference linear solver and establish convergence estimates for this
particular algorithm. This is motivated by the fact that GMRes is nowadays
an established standard Krylov solver for non self-adjoint
systems. However, we expect that similar conclusions also hold for other
linear solvers.

\quad\\ The outline of the present contribution is as follows. In a first
section we fix a specific wave propagation boundary value problem that is
representative of the class of problems we wish to consider. In sections
\ref{sec:funct-spac-oper} and \ref{sec:multi-domain-setting}, we introduce
notations that we will use to reformulate our problem in a multi-domain
setting. In Section \ref{sec:skeleton-formulation}, we describe the DDM
formulation posed on the skeleton of the subdomain partition that will be
the main concern of the present contribution. In Section
\ref{sec:asympt-behav-gmres}, we derive estimates concerning the asymptotic
behaviour of GMRes applied to the skeleton formulation with no
preconditioning. In the last section, we develop a coarse space
preconditioning strategy and derive convergence estimate for 
GMRes applied to the preconditioned system.

\section{Wave propagation problem}

We consider a typical boundary value problem modelling wave propagation in
a bounded cavity.  Let $\Omega\subset \RR^d$ refer to a bounded Lipschitz
open set. Consider a source term $f\in \mL^{2}(\Omega)$, a frequency
parameter $\kappa>0$, and an impedance parameter $\kappa_{*}\geq 0$. We shall
discuss the solution to the following problem
\begin{equation}\label{eq:3}
  \begin{aligned}
    & \text{Find}\;u\in \mH^{1}(\Omega)\;\text{such that}\\
    & \Delta u + \kappa^2 u = - f\quad \text{in}\;\Omega\\
    & \partial_{\bn}u -i\kappa_{*} u = 0\quad \text{on}\;\partial\Omega.
  \end{aligned}
\end{equation}
where $\partial_{\bn}u = \bn\cdot\nabla u$ and $\bn$ is the unit vector
field normal to $\partial\Omega$.  We take \eqref{eq:3} as a prototypical
problem. However what we will describe in the present contribution could
apply to a wider family of problems, following the argumentation provided
in \cite{MR4665035}, see in particular Section 3 of this reference.
Nevertheless, from now on, we will focus on \eqref{eq:3} and discuss in
depth a domain decomposition approach for the solution to this problem.
Problem \eqref{eq:3} can be cast into the following traditional variational
form: find $u\in \mH^{1}(\Omega)$ such that $a(u,v) = \ell(v)\forall v\in
\mH^{1}(\Omega)$ where $\ell(v) = \int_{\Omega} f \overline{v} \;d\bx$ and
\begin{equation}\label{eq:24}
  a(u,v):= \int_{\Omega}\nabla u \nabla \overline{v}
  -\kappa^{2}u\overline{v}\,d\bx -i\kappa_{*}\int_{\partial\Omega}
  u\overline{v} \,d\bs.
\end{equation}
where $d\bs$ refers to the surface measure on $\partial\Omega$ naturally
induced by the volume Lebesgue measure $d\bx$. We will assume that this
problem admits a unique solution which is equivalent to its inf-sup
constant being strictly positive
\begin{equation}\label{eq:14}
  \alpha :=  \mathop{\inf\phantom{p}}_{u\in
    \mH^{1}(\Omega)\setminus\{0\}}\sup_{v\in \mH^{1}(\Omega)\setminus\{0\}}
  \frac{\vert a(u,v)\vert}{\Vert u\Vert_{\mH^{1}(\Omega)} \Vert v\Vert_{\mH^{1}(\Omega)}}>0.
\end{equation}

\section{Function spaces and operators}\label{sec:funct-spac-oper}
We fix here a few notations related to function spaces and
operators. Whenever $\mH$ refers to a Hilbert space over $\CC$, we shall
denote $\mH'$ its dual, and we shall denote $u,p\mapsto \langle u,p\rangle$
the corresponding duality pairing with no complex conjugation involved.
If $\mL:\mH\to\mV$ is a bounded linear maps between two Hilbert spaces,
we shall denote $\mL^{*}:\mV'\to \mH'$ its adjoint, defined by
$\langle \mL(u),p\rangle = \langle u,\mL^*(p)\rangle $ for all $u\in
\mH,p\in\mV'$. 

\quad\\
In a domain $\Omega\subset \RR^d$, we shall consider the Sobolev space
$\mH^{1}(\Omega):= \{v\in \mL^{2}(\Omega),\; \nabla v\in \mL^{2}(\Omega)\}$
equipped with the scalar product $\mA_{\Omega,+}:\mH^{1}(\Omega)\to
\mH^{1}(\Omega)'$ and the corresponding norm defined as
\begin{equation}\label{eq:25}
\begin{aligned}
  &  \langle \mA_{\Omega,+}(u),\overline{v}\rangle:= \int_{\Omega}\nabla
  u\cdot\nabla \overline{v} + \kappa^{2}u\overline{v} \,d\bx\\
  &  \Vert v\Vert_{\mH^{1}(\Omega)}^{2}:= \Vert v\Vert_{\mA_{\Omega,+}}^{2}=
  \Vert \nabla v\Vert_{\mL^{2}(\Omega)}^{2} + \kappa^{2}\Vert v\Vert_{\mL^{2}(\Omega)}^{2}.
\end{aligned}
\end{equation}
Although this is rather obvious, we insist that this scalar product shares
similarities with the bilinear form \eqref{eq:24} of the Helmholtz problem
with the important difference of the + sign in front of the $\kappa^2$
term.  We will denote $\mH^{1}_{0}(\Omega)$ the closure of
$\mathscr{C}^{\infty}_{0}(\Omega):= \{ \varphi\vert_{\Omega},\;
\varphi\in\mathscr{C}^{\infty}(\RR^d),\;\supp(\varphi)\subset \Omega\}$.
The space of Dirichlet traces on the boundary of $\Omega$ forms a Hilbert
space defined by
\begin{equation}
\begin{aligned}
  & \mH^{1/2}(\partial\Omega):= \{
  v\vert_{\partial\Omega},\;v\in\mH^{1}(\Omega)\}\\
  & \Vert v\Vert_{\mH^{1/2}(\partial\Omega)}:=\inf\{\Vert
  v+v_0\Vert_{\mH^{1}(\Omega)}, v_0\in \mH^{1}_0(\Omega)\}.
\end{aligned}
\end{equation}
The norm of the Dirichlet trace space does depend on $\kappa$ through its
relation to $\Vert \cdot\Vert_{\mH^{1}(\Omega)}$.  The Dirichlet trace
operator $\varphi\mapsto \varphi\vert_{\partial\Omega}$ which is well
defined on $\mathscr{C}^{\infty}(\overline{\Omega}):=\{
\varphi\vert_{\Omega},\; \varphi\in\mathscr{C}^{\infty}(\RR^d)\}$, extends
to a bounded surjective map $\mH^{1}(\Omega)\to
\mH^{1/2}(\partial\Omega)$ which we note
\begin{equation}\label{eq:4}
  \begin{aligned}
    & \mB_{\Omega}:\mH^{1}(\Omega)\to \mH^{1/2}(\partial\Omega)\\
    & \mB_{\Omega}(\varphi):= \varphi\vert_{\partial\Omega}\quad\forall
    \varphi\in \mathscr{C}^{\infty}(\overline{\Omega}). 
  \end{aligned}
\end{equation}
We also need to consider the dual space $\mH^{-1/2}(\partial\Omega):=
\mH^{1/2}(\partial\Omega)'$ equipped with the canonical dual norm
\begin{equation}\label{eq:5}
  \Vert p\Vert_{\mH^{-1/2}(\partial\Omega)}:= \sup_{v\in
    \mH^{1/2}(\partial\Omega)\setminus\{0\}}\vert \langle
    p,v\rangle\vert/\Vert v\Vert_{\mH^{1/2}(\partial\Omega)}.
\end{equation}
Denote $\mH^{1}(\Delta,\Omega):= \{ v\in \mH^{1}(\Omega), \Delta
v\in\mL^{2}(\Omega)\}$ equipped with the norm $v\mapsto \Vert \Delta
v\Vert_{\mL^{2}(\Omega)} + \Vert v\Vert_{\mH^{1}(\Omega)}$.  Recall that
the Neumann trace operator $\varphi\mapsto \bn\cdot\nabla
\varphi\vert_{\partial\Omega}$ extends naturally to a bounded surjective
linear map $\mH^{1}(\Delta,\Omega)\to \mH^{-1/2}(\partial\Omega)$.  The
norm \eqref{eq:5} stems from a scalar product which can be described in
terms of the Dirichlet-to-Neumann\footnote{also sometimes referred to as
Stecklov-Poincaré operator} map
$\mT_{\Omega,+}:\mH^{1/2}(\partial\Omega)\to \mH^{-1/2}(\partial\Omega)$
defined by
\begin{equation}
  \begin{aligned}
    & \mT_{\Omega,+}(\phi\vert_{\partial\Omega}):= \bn_{\Omega}\cdot\nabla
    \phi\vert_{\partial\Omega}\\ & \forall \phi\in
    \mH^{1}(\Delta,\Omega)\;\text{satisfying}\\ & -\Delta
    \phi+\kappa^{2}\phi = 0\quad \text{in}\;\Omega.
  \end{aligned}
\end{equation}
where $\bn_{\Omega}$ refers to the unit vector field at $\partial\Omega$
pointing toward the exterior of $\Omega$. With this notation we have the
following identities
\begin{equation}
  \begin{aligned}
    & \Vert v\Vert_{\mH^{+1/2}(\partial\Omega)}^2= \Vert
    v\Vert_{\mT_{\Omega,+}}^{2}:= \langle \mT_{\Omega,+}(v), \overline{v}\rangle \\
    & \Vert p\Vert_{\mH^{-1/2}(\partial\Omega)}^2= \Vert
    p\Vert_{\mT_{\Omega,+}^{-1}}^{2}:= \langle \mT_{\Omega,+}^{-1}(p), \overline{p}\rangle \\
  \end{aligned}
\end{equation}
The scalar products induced by $\mT_{\Omega,+}^{\pm 1}$ shall stand as reference
scalar products for Dirichlet and Neumann traces in the following. Note
that, by its very definition, $\mT_{\Omega,+}^{-1}(p) = \mB_{\Omega}(u)$
where $\mA_{\Omega,+}(u) = \mB_{\Omega}^{*}(p)$. In other words, we have
the identity
\begin{equation}\label{eq:12}
  \mT_{\Omega,+}^{-1} = \mB_{\Omega}\mA_{\Omega,+}^{-1} \mB_{\Omega}^{*}. 
\end{equation}
In the forthcoming sections, we shall notations similar to the ones
introduced above, mutatis mutandis, substituting $\Omega$ by some other
bounded Lipschitz open set $\omega$, and writing $\mH^{\pm
  1/2}(\partial\omega),\mB_{\omega},\mT_{\omega,+},\dots,etc$.

\section{Multi-domain setting}\label{sec:multi-domain-setting}

Next we introduce notations to describe the domain decomposition
formulation we wish to study for Problem \eqref{eq:3}. We need to introduce
a decomposition of the computational domain into non-overlapping Lipschitz
subdomains $\Omega_j\subset \Omega, j=1\dots, \mJ$
\begin{equation}  
  \begin{aligned}
    & \overline{\Omega} = \overline{\Omega}_1\cup \dots\cup \overline{\Omega}_\mJ\\
    & \text{with}\quad \Omega_j\cap \Omega_k = \emptyset\;\text{for}\;j\neq k,\\
    & \text{and}\quad \Sigma:=\partial\Omega_1\cup \dots \cup
    \partial\Omega_\mJ\\
    & \text{and}\quad \Gamma_j := \partial\Omega_j
  \end{aligned}
\end{equation}
Next we introduce a few notations related to function spaces and
operators attached to each subdomain. Natural function spaces related to
the decomposition consist in cartesian products 
\begin{equation}
  \begin{aligned}
    & \mbH(\Omega) :=
    \mH^{1}(\Omega_1)\times\cdots\times\mH^{1}(\Omega_\mJ)\;\text{equipped with}\\
    & \Vert u\Vert_{\mA_+}^{2} = \langle \mA_+(u),\overline{u}\rangle =
    \sum_{j=1,\dots,\mJ}\langle \mA_{\Omega_j,+}(u_j),\overline{u}_j\rangle
  \end{aligned}
\end{equation}
for all tuples $u = (u_1,\dots,u_\mJ)\in \mbH(\Omega)$ where $\mA_+ :=
\mrm{diag}(\mA_{\Omega_1,+},\dots, \mA_{\Omega_\mJ,+})$ is a block diagonal
scalar product $\mA_+: \mbH(\Omega)\to \mbH(\Omega)'$.  The space
$\mbH(\Sigma)$ is naturally equipped with a scalar product obtained by
grouping local scalar products
$\mT_{\Omega_j,+}:\mH^{1/2}(\partial\Omega_j)\to
\mH^{-1/2}(\partial\Omega_j)$ into a block diagonal operator
\begin{equation}
  \begin{aligned}
    & \mbH(\Sigma) :=
    \mH^{\frac{1}{2}}(\partial\Omega_1)\times\cdots\times\mH^{\frac{1}{2}}(\partial\Omega_\mJ)\;
    \text{equipped with}\\
    & \Vert p\Vert_{\mT_+}^{2} = \langle \mT_+(p),\overline{p}\rangle  =
    \sum_{j=1,\dots,\mJ}\langle \mT_{\Omega_j,+}(p_j),\overline{p}_j\rangle.
  \end{aligned}
\end{equation}
for all tuples $p = (p_1,\dots,p_\mJ)\in \mbH(\Sigma)$ where $\mT_+ :=
\mrm{diag}(\mT_{\Omega_1,+},\dots, \mT_{\Omega_\mJ,+})$ is a block diagonal
scalar product $\mT_+: \mbH(\Sigma)\to \mbH(\Sigma)'$.  Next we introduce
the multi-domain trace operator $\mB:\mbH(\Omega)\to \mbH(\Sigma)$ that,
likewise, is obtained by grouping all local Dirichlet trace operators
$\mB_{\Omega_j}:\mH^{1}(\Omega_j)\to \mH^{1/2}(\partial\Omega_j)$ into a
block diagonal one
\begin{equation}
  \mB:= \mrm{diag}(\mB_{\Omega_1},\dots, \mB_{\Omega_\mJ}).  
\end{equation}
By the very definition of the norm that we chose for
$\mH^{1/2}(\partial\Omega_j)$, we have $\Vert \mB\Vert_{\mbH(\Omega)\to
  \mbH(\Sigma)} = 1 = \Vert \mB^*\Vert_{\mbH(\Sigma)'\to
  \mbH(\Omega)'}$. Consider the space of (single valued) trace functions
defined on the skeleton of the partition
\begin{equation}
  \begin{aligned}
    & \mH^{1/2}(\Sigma):= \{u\vert_{\Sigma},\;u\in\mH^{1}(\Omega) \}\\
    & \Vert v\Vert_{\mH^{1/2}(\Sigma)}:= \inf\{\Vert
    u\Vert_{\mH^{1}(\Omega)}, u\vert_{\Sigma} = v\}
  \end{aligned}
\end{equation}
Finally we introduce restriction operators $\mathcal{R}:\mH^{1}(\Omega)\to
\mbH(\Omega)$ and $\mR:\mH^{1/2}(\Sigma)\to \mbH(\Sigma)$ that encode
Dirichlet tranmission conditions by mapping globally defined functions to
tuples formed by their restrictions to subdomains,
\begin{equation}
  \begin{aligned}
    & \mR(v) := (v\vert_{\Gamma_1},\dots,v\vert_{\Gamma_\mJ}),\\
    & \mathcal{R}(u):= (u\vert_{\Omega_1},\dots,u\vert_{\Omega_\mJ}).
  \end{aligned}
\end{equation}

\section{Skeleton formulation}\label{sec:skeleton-formulation}

Based on the multi-domain setting that introduced above, we shall now
derive the domain decomposition under study. We need first to
reformulate \eqref{eq:3}. Define $\mA:\mbH(\Omega)\to \mbH(\Omega)'$ as the
following block-diagonal operator
\begin{equation}\label{eq:8}
  \begin{aligned}
    & \mA = \mrm{diag}(\mA_{\Omega_1},\dots,\mA_{\Omega_\mJ})\quad \text{where}\\
    & \langle \mA_{\Omega_j}(u),v\rangle:= \int_{\Omega_j}\nabla u \nabla v
    -\kappa^{2}u v\,d\bx -i\kappa_{*}\int_{\partial\Omega\cap \partial\Omega_j} u v \,d\bs.
  \end{aligned}
\end{equation}
With this definition, we have $a(u,v) = \langle
\mA\mathcal{R}(u),\mathcal{R}(\overline{v})\rangle$ for all $u,v\in \mH^{1}(\Omega)$.
Similarly define $\boldsymbol{\ell} = (\ell_1,\dots,\ell_{\mJ})\in
\mbH(\Omega)'$ by $\langle \ell_j, v\rangle:= \int_{\Omega_j}f v\,d\bx$, so that $\ell(v) =
\langle\boldsymbol{\ell},\mathcal{R}(v)\rangle$ for all $v\in \mH^{1}(\Omega)$. With these
notations, the boundary value problem \eqref{eq:3} can be rewritten as:
find $u\in\mH^{1}(\Omega)$ such that 
\begin{equation}\label{eq:6}
  \mathcal{R}^*\mA\mathcal{R}(u) = \mathcal{R}^*(\boldsymbol{\ell}).
\end{equation}
Following the analysis presented for example in
\cite{zbMATH07381643,MR4665035}, Equation \eqref{eq:6} can be recast in
terms of boundary impedance traces.  Such a reformulation rests on three
ingredients.

\begin{itemize}
\item[\textbf{a)}] \textbf{Impedance operator} This is a positive definite self-adjoint
  operator $\mT:\mbH(\Sigma)\to \mbH(\Sigma)'$ that only differs from
  $\mT_+$ by a compact perturbation i.e. $\mT-\mT_+$ is compact. As a
  consequence of Fredholm alternative, the norms $p\mapsto \Vert
  p\Vert_{\mT} := \sqrt{\langle \mT(p),\overline{p}\rangle}$ and $p\mapsto
  \Vert p\Vert_{\mT_+}$ are equivalent. We shall also consider the norm
  $\Vert \cdot\Vert_{\mT^{-1}}$ on $\mbH(\Sigma)'$ defined by
  \begin{equation}
    \Vert p\Vert_{\mT^{-1}}^2:= \langle \mT^{-1}(p),\overline{p}\rangle.
  \end{equation}
  We shall further assume that this impedance operator is subdomainwise
  block-diagonal $\mT = \mrm{diag}(\mT_{\Omega_1},\dots,\mT_{\Omega_\mJ})$
  where each $\mT_{\Omega_j}:\mH^{1/2}(\partial\Omega_j)\to
  \mH^{-1/2}(\partial\Omega_j)$ is positive definite self-adjoint and only
  differs from $\mT_{\Omega_j,+}$ by a compact perturbation.
\item[\textbf{b)}] \textbf{Non-local exchange operator} This is a bounded
  map $\Pi:\mbH(\Sigma)'\to \mbH(\Sigma)'$ that maintains a non-local
  coupling between all subdomains and defined by the explicit formula
  \begin{equation}
    \Pi := 2\mT\mR(\mR^*\mT\mR)^{-1}\mR^* -\Id
  \end{equation}
  The operator $\mT\mR(\mR^*\mT\mR)^{-1}\mR^*$ is the $\mT^{-1}$-orthogonal
  projection onto $\mrm{Im}(\mR)$ so that $\Pi^2 = \Id$. An
  important feature is \cite[Prop. 4.3]{MR4665035}
  \begin{equation}
    \Vert \Pi(\bp)\Vert_{\mT^{-1}} = \Vert \bp\Vert_{\mT^{-1}}
    \quad\forall \bp\in\mbH(\Sigma)'
  \end{equation}
  i.e. this is an isometric operator.  A more detailed description of this
  operator and its relation to transmission conditions can be found in
  \cite[Sec.4]{MR4665035}.
  
\item[\textbf{c)}] \textbf{Local scattering operator} This is a bounded map
  $\mS:\mbH(\Sigma)'\to \mbH(\Sigma)'$ that acts independently in each
  subdomain i.e. it is block diagonal $\mS =
  \mrm{diag}(\mS_{\Omega_1},\dots,\mS_{\Omega_\mJ})$ where each
  $\mS_{\Omega_j}: \mH^{-1/2}(\partial\Omega_j)\to
  \mH^{-1/2}(\partial\Omega_j)$ is a bounded map.  Locally, it takes an
  ingoing impedance trace as input, solves the wave equation, and returns
  the corresponding outgoing trace. It is explicitely defined by
  \begin{equation}\label{eq:11}
    \mS := 2i\mT\mB(\mA-i\mB^*\mT\mB)^{-1}\mB^*+\Id
  \end{equation}
  Because each operator of the formula in the right hand side above is
  block diagonal, the scattering map $\mS$ is itself subdomain-wise block
  diagonal, which is a desirable feature in the perspective of large scale
  distributed memory computations. A further important feature is its
  contractivity   \cite[Prop.7.2]{MR4665035}.
  \begin{equation}\label{eq:7}
    \Vert \mS(\bp)\Vert_{\mT^{-1}} \leq \Vert \bp\Vert_{\mT^{-1}}
    \quad\forall \bp\in\mbH(\Sigma)'    
  \end{equation}
  which reflects conservation of energy. In the case
  where $\kappa_*=0$ i.e. pure Neumann boundary condition in \eqref{eq:3},
  the operator $\mA$ given by \eqref{eq:8} is self-adjoint and the
  scattering operator is an isometry $\Vert \mS(\bp)\Vert_{\mT^{-1}} =
  \Vert \bp\Vert_{\mT^{-1}}$. A more detailed description of $\mS$ can
  be found e.g. in \cite[Sec.7]{MR4665035}.
  
\end{itemize}
Defining $\bff:= -2i\Pi\mT\mB(\mA-i\mB^*\mT\mB)^{-1}\boldsymbol{\ell}$, 
the skeleton formulation associated to the domain decomposition strategy
that we wish to investigate now writes
\definecolor{grisclair}{gray}{0.9}
\begin{empheq}[box={\setlength{\fboxsep}{10pt}\colorbox{grisclair}}]{equation}\label{eq:9}
  \begin{aligned}
    & \text{Find}\;\bp\in \mbH(\Sigma)'\;\text{such that}\\
    & (\Id + \Pi\mS)\bp = \bff
  \end{aligned}
\end{empheq}
%% \begin{equation}\label{eq:9}
%%   \begin{aligned}
%%     & \text{Find}\;\bp\in \mbH(\Sigma)'\;\text{such that}\\
%%     & (\Id + \Pi\mS)\bp = \bff
%%   \end{aligned}
%% \end{equation}
The equation above is equivalent to \eqref{eq:8} and hence to the boundary
value problem \eqref{eq:3} in the following sense, see
\cite[Prop.8.1]{MR4665035}.
\begin{lem}\quad\\
  Define $\bff:= -2i\Pi\mT\mB(\mA-i\mB^*\mT\mB)^{-1}\boldsymbol{\ell}$.  A
  tuple of traces $\bp\in \mbH(\Sigma)'$ solves \eqref{eq:9} if and only if
  $\mathcal{R}(u) = (\mA-i\mB^*\mT\mB)^{-1}(\mB^*\bp+\boldsymbol{\ell})$
  where $u$ solves \eqref{eq:3}.
\end{lem}

\noindent 
Formulation \eqref{eq:9} appears well suited to linear solvers because its
operator $\Id + \Pi\mS$ is coercive. Indeed applying \cite[Thm.9.2 \&
  Cor.9.3]{zbMATH07381643}, and adapting the proof of \cite[Cor.8.4 \&
  Prop.10.4]{MR4433119}, yields the following coercivity bound
\begin{equation}\label{eq:26}
  \begin{aligned}
   & \inf_{\bp\in \mbH(\Sigma)'}\frac{\Re \{\langle
      \mT^{-1}(\Id+\Pi\mS)\bp,\overline{\bp}\rangle\}}{
      \Vert \bp\Vert_{\mT^{-1}}^{2}}\geq \frac{\gamma_{\star}^2}{2}\quad\text{where} \\
    \gamma_{\star} := 
    & \inf_{\bp\in \mbH(\Sigma)'}\frac{\Vert
      (\Id+\Pi\mS)\bp\Vert_{\mT^{-1}}}{\Vert \bp\Vert_{\mT^{-1}}}\geq
      \frac{\sqrt{2}\alpha}{(\lambda_{\max})^2+(2\Vert a\Vert/\lambda_{\min})^2}
  \end{aligned}
\end{equation}
In the last bound above, the term $\Vert a\Vert$ refers to the continuity
modulus of the bilinear form $a(\cdot,\cdot)$ with respect to the
$\mH^1$-norm \eqref{eq:25}, the coefficient $\alpha$ is given by
\eqref{eq:14}, and $\lambda_{\max},\lambda_{\min}$ refer to the extreme
eigenvalues
\begin{equation*}
  \lambda_{\max} := \max_{\bv\in\mbH(\Sigma)\setminus\{0\} }\frac{\Vert
    \bv\Vert_{\mT_{\phantom{+}}}}{\Vert \bv\Vert_{\mT_+}}\quad\text{and}\quad
    \lambda_{\min} := \min_{\bv\in\mbH(\Sigma)\setminus\{0\} }\frac{\Vert
  \bv\Vert_{\mT_{\phantom{+}}}}{\Vert \bv\Vert_{\mT_+}}.
\end{equation*}
From Estimate \eqref{eq:26}, since both $\Pi$ and $\mS$ are contractions in
the norm $\Vert \cdot\Vert_{\mT^{-1}}$, their continuity modulus is less
than $1$, which implies that $\gamma_\star^{2}/2\leq 2$, and this rewrites
$\gamma_\star\leq 2$.

\section{Asymptotic behaviour of GMRes}\label{sec:asympt-behav-gmres}

In the present section we wish to discuss an effective solution strategy
for \eqref{eq:9}. The operator $\Id + \Pi\mS$ is coercive but not
self-adjoint, which suggests using a GMRes solver. For a complete
presentation of GMRes, we refer to \cite{zbMATH03967793},
\cite[\S6.5]{Saad2003IMS} and \cite[Chap.2]{zbMATH06385506}.

Of course, in the present case, we are considering an abstract infinite
dimensional problem, whereas most of the literature considers GMRes and
Krylov solvers in a finite dimensional setting. However, we wish to analyze
fundamental trends of Krylov solvers grounded in the properties of the
continuous equation \eqref{eq:9}, leveraging compactness of certain parts
of the formulation, which can only be achieved in an infinite dimensional
setting (in finite dimension, all operators are compact). This is why we
adopt a perspective on GMRes and Krylov solvers close to
\cite{zbMATH05080488,zbMATH00556557}.

\quad\\
When applied to \eqref{eq:9}, taking account of the scalar product induced
by $\mT^{-1}$, GMRes proceeds as follows.  Consider an initial guess
$\bp_0\in \mbH(\Sigma)'$, and denote $\br_0:=\bff-(\Id+\Pi\mS)\bp_0$ the
corresponding residuum.  Denote $\mathscr{K}_{n}(\br_0):=
\mrm{span}_{k=0,\dots,n-1}\{ (\Id+\Pi\mS)^{k}\br_0\}$ the associated Krylov
space of order $n$. With these notations, the $n$-th iteration of GMRes
computes the element $\bp_n\in \bp_0 + \mathscr{K}_{n}(\br_0)$ satisfying
\begin{equation}\label{eq:10}
  \begin{aligned}
    \Vert \bff - (\Id+\Pi\mS)\bp_n\Vert_{\mT^{-1}}
    & = \min_{\bz\in \bp_0 + \mathscr{K}_{n}(\br_0)}\Vert \bff -
    (\Id+\Pi\mS)\bz\Vert_{\mT^{-1}}\\
    & = \min_{Q\in \mathbb{P}_n\lbrack X\rbrack, Q(0) = 1}\Vert Q(\Id+\Pi\mS)\br_0\Vert_{\mT^{-1}}
  \end{aligned}
\end{equation}
where $\mathbb{P}_n\lbrack X\rbrack$ refers to the space of polynomials of
degree $n$. Setting $\br_n = \bff - (\Id+\Pi\mS)\bp_n$, we know that
$\lim_{n\to \infty} \Vert \br_n\Vert_{\mT^{-1}} = 0$ and we can estimate
the rate of decay of the residual by means of \eqref{eq:26}: since
$\Vert (\Id + \Pi\mS)\bp\Vert_{\mT^{-1}}\leq 2\Vert \bp\Vert_{\mT^{-1}}$,
Elman's bound \cite{ElmanPhDThesis,zbMATH03831185} yields
\begin{equation}\label{eq:28}
  \frac{\Vert \br_n\Vert_{\mT^{-1}}}{\Vert \br_0\Vert_{\mT^{-1}}}\leq
  \Big(1-\Big(\frac{\gamma_\star}{2}\Big)^2\Big)^{n}
\end{equation}

\subsection{Superlinear convergence}
Although \eqref{eq:28} already shows that $\lim_{n\to \infty}\Vert
\br_n\Vert_{\mT^{-1}} = 0$ with at least geometric convergence, the
structure of $\Id+\Pi\mS$ can be further exploited to describe the
asymptotic decay of the residual. For that purpose we establish a few
elementary results. As an intermediate step, substitute $\mA$ by
$\mA_{+}$ and $\mT$ by $\mT_+$ in the expression \eqref{eq:11} of the
scattering map. This yields nothing but a simple multiple of the identity.

\begin{lem}\label{sec:expl-deriv-param}\quad\\
  $2i\mT\mB(\mA_{+}-i\mB^*\mT_+\mB)^{-1}\mB^*+\Id = i\Id$.
\end{lem}

\noindent 
\textbf{Proof:}

Take any $\bp\in\mbH(\Sigma)'$ and set $\bu :=
(\mA_{+}-i\mB^*\mT_+\mB)^{-1}\mB^*\bp$ so that $(\mA_{+}-i\mB^*\mT_+\mB)\bu
= \mB^*\bp \iff \mA_{+}\bu = \mB^*(\bp+i\mT_+\mB\bu) \Rightarrow \mB(\bu) =
\mB\mA_{+}^{-1}\mB^*(\bp+i\mT_+\mB\bu)$. Since $\mT_+^{-1} =
\mB\mA_{+}^{-1}\mB^*$ by the property \eqref{eq:12} satisfied by each
$\mT_{\Omega_j,+}$, we conclude that $\mB(\bu) = \mT_+^{-1}(\bp)+i\mB(\bu)$
and this finally leads to $2i\mT_+\mB(\bu) +\bp = \bp+2i\bp/(1-i) = i\bp$.
Coming back to the very definition of $\bu$ yields the desired
identity. \hfill $\Box$

\quad\\
From this we deduce that the scattering map is a compact perturbation of a
multiple of the identity.

\begin{lem}\label{sec:expl-deriv-param-1}\quad\\
  The operator $\mS-i\Id$ is compact as a map $\mbH(\Sigma)'\to\mbH(\Sigma)'$.
\end{lem}

\noindent 
\textbf{Proof:}
Starting from the Formula of Lemma \ref{sec:expl-deriv-param}, a direct
calculation shows that $\mS-i\Id =
2i\mT\mB(\mA-i\mB^*\mT\mB)^{-1}(\mA_{+}-\mA-
i\mB^*(\mT_+-\mT)\mB)(\mA_{+}-i\mB^*\mT_+\mB)^{-1}\mB^*$.  We already know,
by hypothesis, that $\mT_+-\mT$ is compact.  Next, comparing \eqref{eq:25}
with Expression \eqref{eq:8}, we see that
\begin{equation}\label{eq:16}
  \langle (\mA_{\Omega_j,+}-\mA_{\Omega_j})u,v\rangle = 2\kappa^2\int_{\Omega_j}u v\,d\bx
  + i\kappa_{*}\int_{\partial\Omega\cap \partial\Omega_j}uv d\bs
\end{equation}
hence, due to the compact embedding of $\mH^{1}(\Omega_j)$ into
$\mL^{2}(\Omega_j)$, we see that $\mA-\mA_{+}$ is compact as a map
$\mbH(\Omega)\to\mbH(\Omega)'$. This establishes the compactness
result. \hfill $\Box$

\begin{cor}\label{cor:effect-solut-strat}\quad\\
  The operator $(\Id+\Pi\mS)^{4}+4\Id$ is compact as a map $\mbH(\Sigma)'\to\mbH(\Sigma)'$.
\end{cor}
\noindent 
\textbf{Proof:}

Recall that compact operators form a two-sided ideal
\cite[Thm.4.18]{zbMATH01022519}.  According to the previous lemma
$\mS-i\Id$ is compact. As a consequence $(\Id + \Pi\mS) - (\Id + i\Pi)$ and
$(\Id + \Pi\mS)^{4} - (\Id + i\Pi)^{4}$ are also compact operators.
Exploiting the property $\Pi^2 = \Id$, we have $(\Id +i\Pi)^{4} =
(2i\Pi)^{2} = -4\Id$. We conclude that $(\Id + \Pi\mS)^{4} + 4\Id$ is
compact. \hfill $\Box$

\quad\\
The above result can be leveraged to prove superlinear-convergence of
GMRes. For a bounded linear operator $\mL:\mH\to\mV$ between two Hilbert
spaces $\mH,\mV$ equipped with the norms
$\Vert\cdot\Vert_{\mH},\Vert\cdot\Vert_{\mV}$, for any integer $j\geq 1$,
its singular values shall be defined by
\begin{equation}
  \sigma_j(\mL):=\inf\{\Vert
  \mL-\mL_j\Vert_{\mH\to \mV},\mL_j:\mH\to \mV,\;\;\mrm{rank}(\mL_j)<j\}
\end{equation}
A bounded linear operator $\mK:\mH\to \mV$ is compact if and only if its
singular values decrease to zero $\lim_{j\to \infty}\sigma_j(\mK) = 0$
\cite[Thm.2.6.5]{zbMATH00043989}. Now let us point to a theorem by Moret
\cite[Thm.1]{zbMATH01011186}.

\begin{thm}\label{thm:effect-solut-strat}\quad\\
  For any invertible bounded linear operator $\mL:\mH\to \mH$ in a Hilbert
  space $\mH$ equiped with norm $\Vert\cdot\Vert_{\mH}$, and for any
  $\bg\in \mH\setminus\{0\}, \lambda\in\mathbb{C}$, we have 
  \begin{equation}\label{eq:22}
      \min_{Q\in \mathbb{P}_n\lbrack X\rbrack, Q(0) = 1}\frac{\Vert
      Q(L)\bg\Vert_{\mH}}{\Vert\bg\Vert_{\mH}} \leq
      \prod_{j=1}^{n}\sigma_j(\mL-\lambda\Id)\sigma_j(\mL^{-1}).
  \end{equation}
\end{thm}
\noindent 
Similar results can be found in \cite[Chap.5]{zbMATH00556557}. With the
previous theorem, if $\mL-\lambda\Id$ is assumed compact, then we have
$\lim_{j\to \infty}\sigma_j(\mL-\lambda\Id) = 0$, so the left hand side of
the inequality decays faster than $\delta^{n}$, for any $\delta\in
(0,1)$. This is what we call "superlinear convergence".

\quad\\
Taking $\mL = (\Id+\Pi\mS)^4$ and $\lambda = -4$ leads to a situation
fitting Theorem \ref{thm:effect-solut-strat} according to Corollary
\ref{cor:effect-solut-strat}. We can apply \eqref{eq:22} and, for all
$j\geq 0$, we make the crude estimate $\sigma_j(\mL^{-1})\leq \Vert (\Id +
\Pi\mS)^{-4}\Vert_{\mbH(\Sigma)'\to \mbH(\Sigma)'}\leq \gamma_\star^{-4}$, which leads to
\begin{equation*}
  \begin{aligned}
    \frac{\Vert \br_{4n}\Vert_{\mT^{-1}}}{\Vert \br_0\Vert_{\mT^{-1}}}
    & = \min_{Q\in \mathbb{P}_{4n}\lbrack X\rbrack, Q(0) = 1}
    \frac{\Vert Q(\Id+\Pi\mS)\br_0\Vert_{\mT^{-1}}}{\Vert \br_0\Vert_{\mT^{-1}}}\\
    & \leq  \prod_{1\leq j \leq n}
    \Big(\frac{\sigma_j(4\Id +
    (\Id + \Pi\mS)^4)}{\gamma_\star^{4}}\Big).
  \end{aligned}
\end{equation*}
We deduce a convergence bound for the whole sequence $\{\br_n\}_{n\geq 0}$
by simple euclidean division. For any $n\in\mathbb{N}$, let
$\floor{n/4}\in\mathbb{N}$ refer to the unique integer satisfying
$\floor{n/4}\leq n/4 \leq \floor{n/4}+1$. By the very definition of the
GMRes residual as a minimizer, $\Vert \br_{n}\Vert_{\mT^{-1}}\leq \Vert
\br_{4\floor{n/4}}\Vert_{\mT^{-1}}$ hence,
\begin{equation}\label{eq:27}
  \frac{\Vert \br_{n}\Vert_{\mT^{-1}}}{\Vert \br_0\Vert_{\mT^{-1}}} \leq  \prod_{1\leq j \leq \floor{n/4}}
  \Big(\frac{\sigma_j(4\Id +
    (\Id + \Pi\mS)^4)}{\gamma_\star^{4}}\Big).
\end{equation}
Due to the compactness of $4\Id + (\Id + \Pi\mS)^4$ we have $\lim_{j\to
  \infty}\sigma_j(4\Id + (\Id + \Pi\mS)^4)/\gamma_\star^{4} = 0$, which
yields superlinear convergence of GMRes applied to \eqref{eq:9}. As a
consequence, asymptotically for $n\to \infty$, the bound \eqref{eq:27} is
much sharper than Elman's estimate \eqref{eq:28}.

\subsection{Singular value decay}
In practice, based on the boundary value problem \eqref{eq:3} under
consideration, we may have sharper information about the compactness of
$\mS-i\Id$, and this can be exploited to quantify even more precisely the
convergence of GMRes.  For any $p\in\lbr 1,+\infty)$ and two
Hilbert spaces $\mH,\mV$, the Schatten $p$-class is defined by
\begin{equation}
  \mathcal{C}_p(\mH,\mV):=\{\mK:\mH\to \mV\;\text{compact},\; \Vert
  \mK\Vert_{\mathcal{C}_p}:= \big(\sum_{j=0}^{+\infty}\sigma_j(\mK)^{p}\big)^{1/p}<+\infty\}
\end{equation}
see \cite[Chap.11]{zbMATH00043989}, \cite[Chap.2]{zbMATH03354068} for
further details on this class of compact operators. The norm $\Vert
\cdot\Vert_{\mathcal{C}_p}$ quantifies the rate of decay of the singular
values $\sigma_j(\mK)$. Following a remark from
\cite{zbMATH03700119,zbMATH01011186}, this can be used to further estimate
the right hand side in \eqref{eq:22} by means of the following lemma.

\begin{lem}\label{sec:superl-conv-gmres-1}\quad\\
  For $\mK\in \mathcal{C}_p(\mH,\mV), p\in \lbr 1,+\infty)$, we have
  $\prod_{j=1}^{n}\sigma_j(\mK)\leq (n^{-1/p}\Vert \mK\Vert_{\mathcal{C}_p})^{n}$
\end{lem}
\noindent \textbf{Proof:}

Using the inequality between arithmetic and geometric means,
e.g. \cite[Chap.3]{zbMATH01022658}, and H\"older inequality, we have
%\begin{equation}
$\prod_{j=1}^{n}\sigma_j(\mK)^{1/n}\leq
\sum_{j=1}^{n}\sigma_j(\mK)/n\leq
n^{-1/p}(\sum_{j=1}^{n}\sigma_j(\mK)^p)^{1/p} \leq  n^{-1/p} \Vert \mK\Vert_{\mathcal{C}_p}$.
%\end{equation}
\hfill $\Box$

\quad\\
The next lemma gives sufficient condition for the operator $(\Id +
\Pi\mS)^4+4\Id$ to belong to the Schatten $p$-class with a quantitative
norm estimate.

\begin{lem}\label{sec:superl-conv-gmres}\quad\\
  Assume that $\mA-\mA_+:\mbH(\Omega)\to \mbH^1(\Omega)'$ and
  $\mT-\mT_+:\mbH(\Sigma)\to \mbH(\Sigma)'$ belong to the Schatten $p$-class
  for some $p\geq 1$. Then $\mS-i\Id$ and $(\Id+\Pi\mS)^{4}+4\Id$ belong to the Schatten
  $p$-class with
  \begin{equation}
    \begin{aligned}
      & \Vert \mS-i\Id\Vert_{\mathcal{C}_p}\leq \sqrt{2}\mathcal{E}\\
      & \Vert 4\Id+(\Id+\Pi\mS)^{4}\Vert_{\mathcal{C}_p}\leq
      C(\mathcal{E})\mathcal{E}\\[10pt]
      & \text{where}\quad \mathcal{E} = (\Vert\mA-\mA_+\Vert_{\mathcal{C}_p} +
      \Vert \mT-\mT_+\Vert_{\mathcal{C}_p})/(\sqrt{2}\beta)\\
      & \textcolor{white}{where}\quad 1/\beta = \Vert(\mA-i\mB^*\mT\mB)^{-1}\Vert_{\mbH(\Omega)'\to\mbH(\Omega)}\\
      & \textcolor{white}{where}\quad C(\mathcal{E}) =
      4(4+6\mathcal{E} + 4\mathcal{E}^2 + \mathcal{E}^3)
    \end{aligned}
  \end{equation}
\end{lem}
\noindent
\textbf{Proof:}

The operartor $\mA-\mA_+-i\mB^*(\mT-\mT_+)\mB$ belongs to the schatten
$p$-class with $\Vert
\mA-\mA_+-i\mB^*(\mT-\mT_+)\mB\Vert_{\mathcal{C}_p}\leq \Vert
\mA-\mA_+\Vert_{\mathcal{C}_p} + \Vert \mT-\mT_+\Vert_{\mathcal{C}_p}$
since $\Vert \mB\Vert_{\mbH(\Omega)\to \mbH(\Sigma)} = 1$. A direct
computation exploiting the positivity of $\mA_+,\mT_+$ reveals that
$\Vert(\mA_{+}-i\mB^*\mT_+\mB)^{-1}\Vert_{\mbH(\Omega)'\to
  \mbH(\Omega)}\leq 1$. From the explicit expression of $\mS-i\Id$ derived
in the proof of Lemma \ref{sec:expl-deriv-param-1}, we conclude that
$\mS-i\Id$ also belongs to the Schatten $p$-class with
\begin{equation}
  \begin{aligned}
    & \Vert \mS-i\Id\Vert_{\mathcal{C}_p}\leq (\Vert
    \mA-\mA_+\Vert_{\mathcal{C}_p} + \Vert
    \mT-\mT_+\Vert_{\mathcal{C}_p})/\beta\\
    & \text{with}\;\;1/\beta :=
    \Vert(\mA-i\mB^*\mT\mB)^{-1}\Vert_{\mbH(\Omega)'\to\mbH(\Omega)}
  \end{aligned}
\end{equation}
The coefficient $\beta$ is the inf-sup constant of local subproblems
modeled by the operator $\mA-i\mB^*\mT\mB$. Recall $\Pi^2 = \Id$, which
yields $(\Id + i\Pi)^{-1} =(\Id-i\Pi)/2$ and $(\Id + i\Pi)^4 =
-4\Id$. Hence we can write $(\Id + \Pi\mS)^{4} = ((\Id + i\Pi) +
\Pi(\mS-i\Id))^{4} = -4 (\Id + (\Pi-i\Id)(\mS-i\Id)/2)^{4}$ which finally
leads to 
\begin{equation}
  \begin{aligned}
    & (\Id + \Pi\mS)^{4} + 4\Id = - C(X) X\\
    & \text{with}\quad X = (\Pi-i\Id)(\mS-i\Id)/2. 
  \end{aligned}
\end{equation}
Since $\Pi$ is self-adjoint with respect to $\mT^{-1}$, and $\Pi^2 = \Id$, we see
that $(\Pi-i\Id)/\sqrt{2}$ is unitary with respect to $\Vert
\cdot\Vert_{\mT^{-1}}$. As a consequence, with the above notations, we have
$\Vert X\Vert_{\mathcal{C}_p}\leq \Vert
\mS-i\Id\Vert_{\mathcal{C}_p}/\sqrt{2}$ which concludes the proof.  \hfill $\Box$

\quad\\
Combining \eqref{eq:27}, and Lemma \ref{sec:superl-conv-gmres-1} and
\ref{sec:superl-conv-gmres}, yields the following convergence result when
applying GMRes to solve \eqref{eq:9}. 

\begin{cor}\quad\\
  Assume that $\mA-\mA_+:\mbH(\Omega)\to \mbH^1(\Omega)'$ and
  $\mT-\mT_+:\mbH(\Sigma)\to \mbH(\Sigma)'$ belong to the Schatten
  $p$-class for some $p\geq 1$. If $\bp_n\in \mbH(\Sigma)'$ refers to the
  $n$-th iterate obtained by applying GMRes to \eqref{eq:9}, setting
  $\br_n = \bff - (\Id +\Pi\mS)\bp_n$ we have 
  \begin{equation}\label{eq:29}
    \frac{\Vert \br_n\Vert_{\mT^{-1}}}{\Vert \br_0\Vert_{\mT^{-1}}}\leq
    \Big(\frac{\Vert 4\Id + (\Id+\Pi\mS)^{4}\Vert_{\mathcal{C}_p}}{\gamma_\star^{4}\,\floor{n/4}^{1/p}}\Big)^{\floor{n/4}}
  \end{equation}
\end{cor}

\quad\\
Let us underline that, if $\mT=\mT_+$, then the explicit expression
\eqref{eq:24} can be used to show that $\mA-\mA_+$ belongs to a Schatten
$p$-class for some $p\geq 1$, making use of Weyl's law for the Neumann
laplacian in domains with Lipschitz boundary (see
e.g. \cite[Corollary 1.6]{zbMATH02196648}), and Weyl asymptotics of
Poincaré-Steklov eigenvalues (see e.g. \cite[Theorem 2.1]{zbMATH07797566}).

\section{Preconditioning strategy}

Estimates \eqref{eq:28}, \eqref{eq:27}, \eqref{eq:29} indicate that such a
standard Krylov solver as GMRes is systematically convergent with a
favorable asymptotic behaviour for a growing dimension of Krylov
spaces. Unfortunately, this is not a garantee of effective computational
performance. Indeed, all these estimates dramatically depend on the inf-sup
constant $\gamma_\star = \Vert (\Id+\Pi\mS)^{-1}\Vert_{\mbH(\Sigma)'\to
  \mbH(\Sigma)'}^{-1}$ from \eqref{eq:26}, and they deteriorate
as $\gamma_\star$ gets close to zero, which occurs for example in the high
frequency regime $\kappa\to \infty$, or whenever the geometrical
configuration involves trapping features. It is thus desirable to devise a
preconditioning strategy that stabilizes the convergence of linear solvers,
no matter the frequency regime or the geometrical configuration. The main
goal of the present contribution is to discuss an efficient solution
strategy for \eqref{eq:9} that achieves such a robustness.

\subsection{First order preconditioner}\label{sec:first-order-prec}
As we have seen, the operator $\Id+\Pi\mS$ differs from  $\Id+i\Pi$ modulo
a compact perturbation. As a consequence $(\Id+i\Pi)^{-1} = (\Id-i\Pi)/2$,
appears as a natural preconditioner. Left multiplying $\Id+\Pi\mS$ by this
operator, and taking account of the identity $\Pi^2 = \Id$, we obtain 
\begin{equation}\label{eq:30}
  \begin{aligned}
    \frac{1}{2}(\Id-i\Pi)(\Id+\Pi\mS)
    & = \frac{1}{2}(\Id-i\Pi)\big( (\Id+i\Pi)+\Pi(\mS-i\Id)\big)\\
    & = \Id + \Big(\frac{\Pi-i\Id}{\sqrt{2}}\Big)\Big(\frac{\mS-i\Id}{\sqrt{2}}\Big)
  \end{aligned}
\end{equation}
Since $\mS-i\Id$ is compact, the operator above appears to be a compact
perturbation of the identity. The corresponding preconditioned variant of
Equation \eqref{eq:9} then writes
\begin{empheq}[box={\setlength{\fboxsep}{10pt}\colorbox{grisclair}}]{equation}\label{eq:31}
  \begin{aligned}
    & \text{Find}\;\bp\in \mbH(\Sigma)'\;\text{such that}\\
    & ( \Id + (\Pi-i\Id)(\mS-i\Id)/2)\bp = (\Id-i\Pi)\bff/2.
  \end{aligned}
\end{empheq}
%% \begin{equation}\label{eq:31}
%%   \begin{aligned}
%%     & \text{Find}\;\bp\in \mbH(\Sigma)'\;\text{such that}\\
%%     & ( \Id + (\Pi-i\Id)(\mS-i\Id)/2)\bp = (\Id-i\Pi)\bff/2.
%%   \end{aligned}
%% \end{equation}
Since $\Pi$ is self-adjoint in the scalar
product $\mT^{-1}$, the adjoint of $(\Pi-i\Id)/\sqrt{2}$ with respect to
$\mT^{-1}$ is $(\Pi+i\Id)/\sqrt{2}$ and a direct calculation shows that it
is its inverse $\lbrack (\Pi-i\Id)/\sqrt{2}\rbrack\cdot \lbrack
(\Pi+i\Id)/\sqrt{2}\rbrack = \Id$, hence $(\Pi-i\Id)/\sqrt{2}$ is unitary
in the norm $\Vert \cdot\Vert_{\mT^{-1}}$. We also underline that, contrary
to $\Id+\Pi\mS$, the operator \eqref{eq:30} is not a priori
$\mT^{-1}$-coercive. On the other hand, because  $(\Pi-i\Id)/\sqrt{2}$  is
a $\mT^{-1}$-isometry, we have
\begin{equation}
  \begin{aligned}
    &\inf_{\bp\in \mbH(\Sigma)'\setminus\{0\}}
    \frac{\Vert \frac{1}{2} (\Id-i\Pi)(\Id+\Pi\mS)\bp\Vert_{\mT^{-1}}}{\Vert
      \bp\Vert_{\mT^{-1}}}\\
    & = \frac{1}{\sqrt{2}}
    \inf_{\bp\in \mbH(\Sigma)'\setminus\{0\}}
    \frac{\Vert(\Id+\Pi\mS)\bp\Vert_{\mT^{-1}}}{\Vert
      \bp\Vert_{\mT^{-1}}} = \frac{\gamma_{\star}}{\sqrt{2}}
  \end{aligned}
\end{equation}
Again because $(\Pi-i\Id)/\sqrt{2}$ is a $\mT^{-1}$-isometry, we have $
\sigma_j((\Pi-i\Id)(\mS-i\Id)/2)\leq \sigma_{j}(\mS-i\Id)/\sqrt{2}$.  Then,
we can again apply the result of Moret i.e Theorem
\ref{thm:effect-solut-strat} with $\mL = (\Id-i\Pi)(\Id+\Pi\mS)/2$ and
$\lambda = 1$ which yields the following estimate.
\begin{lem}\quad\\
  If $\bp_n\in \mbH(\Sigma)'$ refers to the $n$-th iterate obtained by
  applying GMRes to Equation \eqref{eq:31}, setting $\br_n = \bff -
  (\Id +\Pi\mS)\bp_n$ we have
  \begin{equation}\label{eq:1}
    \frac{\Vert \br_n\Vert_{\mT^{-1}}}{\Vert \br_0\Vert_{\mT^{-1}}}\leq
    \prod_{1\leq j\leq n}\frac{\sigma_j(\mS-i\Id)}{\gamma_\star}
  \end{equation}
\end{lem}

\quad\\
Here we also have $\lim_{n\to \infty}\sigma_{n}(\mS-i\Id)/\gamma_\star = 0$
which yields superlinear convergence of GMRes. But this time, compared to
\eqref{eq:27}, only $\gamma_\star$ is involved instead of
$\gamma_\star^{4}$ which suggest that applying GMRes to \eqref{eq:31} leads
to a convergence that is less sensitive to a deterioration of the inf-sup
constant $\gamma_\star$. As a natural corollary, we have a sharper estimate
when compact perturbations are in a Schatten $p$-class. 

\begin{cor}\quad\\
  Assume that $\mA-\mA_+:\mbH(\Omega)\to \mbH^1(\Omega)'$ and
  $\mT-\mT_+:\mbH(\Sigma)\to \mbH(\Sigma)'$ belong to the Schatten
  $p$-class for some $p\geq 1$. If $\bp_n\in \mbH(\Sigma)'$ refers to the
  $n$-th iterate obtained by applying GMRes to \eqref{eq:9}, setting
  $\br_n = \bff - (\Id +\Pi\mS)\bp_n$ we have 
  \begin{equation}\label{eq:2}
    \frac{\Vert \br_n\Vert_{\mT^{-1}}}{\Vert \br_0\Vert_{\mT^{-1}}}\leq
    \Big(\frac{\Vert \mS-i\Id\Vert_{\mathcal{C}_p}}{\gamma_\star\,n^{1/p}}\Big)^n
  \end{equation}
\end{cor}
\quad\\
Estimate \eqref{eq:2} still depends on $\gamma_\star$. A way to filter out
this dependency and achieve robustness relies on a coarse space
construction.

\subsection{Refined preconditioner}
In the present section, we elaborate a refined variant of the
preconditioner studied in the previous section. The preconditioner will be
obtained by replacing $\mS$ by an approximation $\mS_r$ finer than $i\Id$.
The following bound quantifies how close to identity is the preconditioned
system.

\begin{lem}\label{lem:refin-prec}\quad\\
Consider any linear operator $\mS_r:\mbH(\Sigma)'\to \mbH(\Sigma)'$, denote
$\mu_r := \Vert \mS-\mS_r\Vert_{\mbH(\Sigma)'\to \mbH(\Sigma)'}$ and assume
that $\mu_r<\gamma_{\star}$. Then $\Id+\Pi\mS_r$ is invertible and
  \begin{equation}
    \begin{aligned}
      & \Vert \Id - (\Id+\Pi\mS_r)^{-1}(\Id+\Pi\mS)\Vert_{\mbH(\Sigma)'\to
        \mbH(\Sigma)'}\leq \mu_{r}/(\gamma_\star-\mu_r)\\
      & \Vert (\Id+\Pi\mS)^{-1}(\Id+\Pi\mS_r)\Vert_{\mbH(\Sigma)'\to
        \mbH(\Sigma)'}\leq 1 + \mu_r/\gamma_\star
    \end{aligned}
  \end{equation}
\end{lem}
\noindent \textbf{Proof:}

We have $\Id - (\Id+\Pi\mS_r)^{-1}(\Id+\Pi\mS) =
(\Id+\Pi\mS_r)^{-1}\Pi(\mS-\mS_r)$ and, since $\Pi$ is a
$\mT^{-1}$-isometry, we obtain $\Vert \Id -
(\Id+\Pi\mS_r)^{-1}(\Id+\Pi\mS)\Vert_{\mbH(\Sigma)'\to \mbH(\Sigma)'}\leq
\mu_{r}\Vert(\Id+\Pi\mS_r)^{-1}\Vert_{\mbH(\Sigma)'\to
  \mbH(\Sigma)'}$. Then, for any $\bp\in \mbH(\Sigma)'$ we have $\Vert
(\Id+\Pi\mS_r)\bp\Vert_{\mT^{-1}} \geq \Vert
(\Id+\Pi\mS)\bp\Vert_{\mT^{-1}}- \Vert (\mS-\mS_r)\bp\Vert_{\mT^{-1}}\geq
(\gamma_\star-\mu_r)\Vert \bp\Vert_{\mT^{-1}}$. Dividing this inequality by
$\Vert \bp\Vert_{\mT^{-1}}$ and taking the infimum yields
$\Vert(\Id+\Pi\mS_r)^{-1}\Vert_{\mbH(\Sigma)'\to\mbH(\Sigma)'}\leq
1/(\gamma_\star-\mu_r)$. This yields the first estimate.
To obtain the second estimate, just write
$(\Id+\Pi\mS)^{-1}(\Id+\Pi\mS_r) = \Id + (\Id+\Pi\mS)^{-1}\Pi(\mS_r-\mS)$
and use the definition $\gamma_\star:= \Vert
(\Id+\Pi\mS)^{-1}\Vert_{\mbH(\Sigma)'\to \mbH(\Sigma)'}$. 
\hfill $\Box$

\quad\\
We propose to define the approximate scattering operator by $\mS_r =
i\Id+\mK_{r}$ where $\mK_{r}$ is a rank $r$ approximation of $\mS-i\Id$. An
optimal choice would consist in taking $\mK_{r}$ as the rank $r$ truncated
singular value decomposition of $\mS-i\Id$ which is the unique rank $r$
operator satisfying
\begin{equation}\label{eq:36}
  \Vert \mS-i\Id-\mK_r\Vert_{\mbH(\Sigma)'\to \mbH(\Sigma)'} = \sigma_{r+1}(\mS-i\Id).
\end{equation}
Recall that $\lim_{r\to \infty}\sigma_{r+1}(\mS-i\Id) = 0$ since $\mS-i\Id$
is compact, so the condition $\mu_r<\gamma_{\star}$ from Lemma
\ref{lem:refin-prec} will be fulfilled for $r$ sufficiently large.

With the construction above, where $\mS_r-i\Id$ has finite rank $r$, the
action of $(\Id+\Pi\mS_r)^{-1}$ may be computed as follows. Proceed as in
\S\ref{sec:first-order-prec} to establish that $(\Id -
i\Pi)(\Id+\Pi\mS_r)/2 = \Id+(\Pi-i\Id)(\mS_r-i\Id)/2$. From this and
\eqref{eq:30} we obtain the following identity
\begin{equation}\label{eq:33}
  \begin{aligned}
    & (\Id+\Pi\mS_r)^{-1}(\Id+\Pi\mS) =\\
    & \big(\Id+(\Pi-i\Id)(\mS_r-i\Id)/2\big)^{-1}  \big(\Id+(\Pi-i\Id)(\mS-i\Id)/2\big)
  \end{aligned}
\end{equation}
Next, since $\mS_r-i\Id$ has rank $r$ by our very construction, the
operator $(\Pi-i\Id)(\mS_r-i\Id)/2$ has rank $r$ as well, and
$\Id+(\Pi-i\Id)(\mS_r-i\Id)/2$ can be inverted by means of the Woodbury
formula \cite{MR38136,zbMATH04098606}, \cite[\S2.1.4]{zbMATH06159604}
\begin{equation}\label{eq:32}
  \begin{aligned}
    & (\Id+(\Pi-i\Id)(\mS_r-i\Id)/2)^{-1} = \Id -
    \mU_r(\Id+\mV_r\mU_r)^{-1}\mV_r\\ & \text{where}\;\;
    (\Pi-i\Id)(\mS_r-i\Id)/2 = \mU_r\cdot\mV_r
  \end{aligned}
\end{equation}
where $\mU_r:\CC^r\to \mbH(\Sigma)'$ and $\mV_r:\mbH(\Sigma)'\to \CC^r$ are
bounded maps of rank $r$. In Equation \eqref{eq:32}, we underline that
$\Id+\mV_r\mU_r\in \CC^{r\times r}$ so that, for a moderate value of $r$,
evaluating $(\Id+\mV_r\mU_r)^{-1}$ is computationally reasonnable.  Note
also that, for $\mu_r<\gamma_\star$, the operator
$\Id+(\Pi-i\Id)(\mS_r-i\Id)/2 = \Id+\mU_r\mV_r$ is invertible and, in this
case, the operator $\Id+\mV_r\mU_r$ is also invertible with
$(\Id+\mV_r\mU_r)^{-1} = \Id - \mV_r(\Id+\mU_r\mV_r)^{-1}\mU_r$ according
to Woodbury formula.

\quad\\
Using \eqref{eq:33} and \eqref{eq:32} to left multiply \eqref{eq:9} by
$(\Id + \Pi\mS_r)^{-1}$, leads to the preconditioned formulation
\begin{empheq}[box={\setlength{\fboxsep}{10pt}\colorbox{grisclair}}]{equation}\label{eq:34}
  \begin{aligned}
    & \text{Find}\;\bp\in \mbH(\Sigma)'\;\text{such that}\\
    & (\Id + \Pi\mS_r)^{-1}(\Id + \Pi\mS)\bp = (\Id + \Pi\mS_r)^{-1}\bff
  \end{aligned}
\end{empheq}
%% \begin{equation}\label{eq:34}
%%   \begin{aligned}
%%     & \text{Find}\;\bp\in \mbH(\Sigma)'\;\text{such that}\\
%%     & (\Id + \Pi\mS_r)^{-1}(\Id + \Pi\mS)\bp = (\Id + \Pi\mS_r)^{-1}\bff
%%   \end{aligned}
%% \end{equation}
This is the final formulation that we propose to consider for actual
solution of \eqref{eq:9}. We have the following convergence result for
GMRes.

\begin{prop}\quad\\
  Consider any bounded linear operator $\mS_r:\mbH(\Sigma)'\to
  \mbH(\Sigma)'$ and assume that $\sigma_1(\mS-\mS_r)<\gamma_{\star}$.  If $\bp_n\in
  \mbH(\Sigma)'$ refers to the $n$-th iterate obtained by applying GMRes to
  Equation \eqref{eq:34}, setting $\br_n = \bff-(\Id +\Pi\mS)\bp_n$ we have
  \begin{equation}\label{eq:35}
    \frac{\Vert \br_n\Vert_{\mT^{-1}}}{\Vert \br_0\Vert_{\mT^{-1}}}\leq
    \frac{2+\sigma_1(\mS-\mS_r)}{\gamma_\star-\sigma_1(\mS-\mS_r)} \prod_{1\leq j\leq
      n}\frac{\sigma_j(\mS-\mS_r)}{\gamma_\star}
    \Big(1+\frac{\sigma_j(\mS-\mS_r)}{\gamma_\star}\Big)
    \Big(1-\frac{\sigma_1(\mS-\mS_r)}{\gamma_\star}\Big)^{-1}
  \end{equation}
\end{prop}

\noindent
\textbf{Proof:}

Denote $\mu_r = \sigma_1(\mS-\mS_r) = \Vert
\mS-\mS_r\Vert_{\mbH(\Sigma)'\to\mbH(\Sigma)'}$. Set $\tilde{\br}_n:= (\Id
+ \Pi\mS_r)^{-1}(\bff - (\Id +\Pi\mS)\bp_n) = (\Id + \Pi\mS_r)^{-1}\br_n$
which are the remainders stemming from GMRes applied to Equation
\eqref{eq:34}. For any $\bp\in \mbH(\Sigma)'$, we have
$(\gamma_\star-\mu_r)\Vert \bp\Vert_{\mT^{-1}}\leq \Vert (\Id +
\Pi\mS_r)\bp\Vert_{\mT^{-1}}\leq (2+\mu_r)\Vert \bp\Vert_{\mT^{-1}}$ hence,
taking $\bp = \tilde{\br}_n$, we deduce
\begin{equation*}
  \frac{\Vert \br_n\Vert_{\mT^{-1}}}{\Vert \br_0\Vert_{\mT^{-1}}}\leq
  \Big(\frac{2+\mu_r}{\gamma_\star-\mu_r}\Big)\frac{\Vert
    \tilde{\br}_n\Vert_{\mT^{-1}}}{\Vert \tilde{\br}_0\Vert_{\mT^{-1}}}
\end{equation*}
Next, let us apply Theorem \ref{thm:effect-solut-strat} to Equation
\eqref{eq:34} with $L = (\Id + \Pi\mS_r)^{-1}(\Id + \Pi\mS)$ and $\lambda =
1$. Then we have $L-\Id = (\Id + \Pi\mS_r)^{-1}\Pi (\mS-\mS_r)$ so we
deduce $\sigma_j(L-\Id)\leq \sigma_{j}(\mS-\mS_r)\Vert (\Id +
\Pi\mS_r)^{-1}\Vert_{\mbH(\Sigma)'\to \mbH(\Sigma)'}\leq
\sigma_{j}(\mS-\mS_r)/(\gamma_\star-\mu_r)$. On the other hand, we also
have $L^{-1} = (\Id + \Pi\mS)^{-1}(\Id + \Pi\mS_r) = \Id + (\Id +
\Pi\mS)^{-1}\Pi(\mS_r-\mS)$ so we have $\sigma_j(L^{-1}) \leq 1+
\sigma_j(\mS-\mS_r) \Vert (\Id + \Pi\mS)^{-1}\Vert_{\mbH(\Sigma)'\to
  \mbH(\Sigma)'}\leq 1+\sigma_j(\mS-\mS_r)/\gamma_\star$ which leads to the
desired result. \hfill $\Box$

\quad\\ It is clear from the above discussion, and in particular from
\eqref{eq:35}, that choosing $\mS_r$ as close as possible to $\mS$ will
improve convergence of GMRes applied to the preconditioned equation
\eqref{eq:34}. Assuming that $\mS_r-i\Id$ has rank $r$, the best choice
that can be made is obtained by defining $\mS_r-i\Id$ as the truncated
rank-$r$ singular value decomposition of $\mS-i\Id$ i.e. by setting
$\mS_r:=i\Id+\mK_r$ where $\mK_r$ satisfies \eqref{eq:36}. In this case, we
have $\sigma_{j}(\mS-\mS_r) = \sigma_{r+j}(\mS-i\Id)$ and Estimate
\eqref{eq:35} writes
\begin{equation*}
  \begin{aligned}
    &\frac{\Vert \br_n\Vert_{\mT^{-1}}}{\Vert
      \br_0\Vert_{\mT^{-1}}}\leq\\ &\frac{2+\sigma_{r+1}(\mS-i\Id)}{\gamma_\star-\sigma_{r+1}(\mS-i\Id)}
    \prod_{1\leq j\leq n}\frac{\sigma_{r+j}(\mS-i\Id)}{\gamma_\star}
    \Big(1+\frac{\sigma_{r+j}(\mS-i\Id)}{\gamma_\star}\Big)\Big(1-\frac{\sigma_{r+1}(\mS-i\Id)}{\gamma_\star}\Big)^{-1}
  \end{aligned}
\end{equation*}
This estimate is more complicated than \eqref{eq:2} but appears more
favorable because it involves $\sigma_{r+j}(\mS-i\Id)$ instead of
$\sigma_{j}(\mS-i\Id)$. The parameter $r$ can be adjusted so as to
compensate a deterioration of the decay rate stemming from $\gamma_\star$
approaching $0$. For example, choosing $r$ sufficiently large so that  
$\sigma_{r+1}(\mS-i\Id)\leq \gamma_\star/2$ leads to
$\Vert \br_n\Vert_{\mT^{-1}}/\Vert \br_0\Vert_{\mT^{-1}}\leq
(1+2/\gamma_\star)\Pi_{j=1}^{n}(3\sigma_{r+j}(\mS-i\Id)/\gamma_{\star})$.

\quad\\
In the case where $\mS-i\Id$ belongs to some Schatten $p$-class, the
estimate above leads to enhanced superlinear convergence. The following
corollary is readily obtained by combining Equation \eqref{eq:35} with
Lemma \ref{sec:superl-conv-gmres-1}.
\begin{cor}\quad\\
  Assume that $\mA-\mA_+:\mbH(\Omega)\to \mbH^1(\Omega)'$ and
  $\mT-\mT_+:\mbH(\Sigma)\to \mbH(\Sigma)'$ belong to the Schatten
  $p$-class for some $p\geq 1$. If $\bp_n\in \mbH(\Sigma)'$ refers to the
  $n$-th iterate obtained by applying GMRes to \eqref{eq:34}, setting
  $\br_n = \bff - (\Id +\Pi\mS)\bp_n$ we have
  \begin{equation*}
    \frac{\Vert \br_n\Vert_{\mT^{-1}}}{\Vert \br_0\Vert_{\mT^{-1}}}\leq
    \Big(\frac{2+\sigma_1(\mS-\mS_r)}{\gamma_\star-\sigma_1(\mS-\mS_r)}\Big)
    \Big(\big(\frac{\Vert
      \mS-\mS_r\Vert_{\mathcal{C}_p}}{\gamma_\star\,n^{1/p}}\big) +
    \big(\frac{\Vert
      \mS-\mS_r\Vert_{\mathcal{C}_p}}{\gamma_\star\,n^{1/p}}\big)^{2}\Big)^n
    \Big(1-\frac{\sigma_1(\mS-\mS_r)}{\gamma_\star}\Big)^{-n}.
  \end{equation*}
\end{cor}

\bibliography{manuscrit}
\bibliographystyle{plain}

\end{document}